\documentclass[reqno]{amsart}
\usepackage{amsmath,amssymb,amsthm,enumitem,multirow}
\DeclareMathOperator{\ord}{ord}

\renewcommand{\geq}{\geqslant}
\renewcommand{\leq}{\leqslant}

\newcommand{\bbn}{\mathbb{N}}
\newcommand{\bbp}{\mathbb{P}}
\newcommand{\bbz}{\mathbb{Z}}
\newcommand{\basis}{\mathcal{B}}
\newcommand{\vct}[2]{\langle #1,#2 \rangle}
\newcommand{\compvec}[3]{\vec{#1}^{\vct{#2}{#3}}}
\newcommand{\subgp}[2]{#1\bbz \oplus #2\bbz}
\newcommand{\qtgp}[2]{(\bbz/#1\bbz) \oplus (\bbz/#2\bbz)}
\newcommand{\ds}{\displaystyle}
\newcommand{\separate}{\medskip}

\theoremstyle{plain}
\newtheorem{thm}{Theorem}
\newtheorem*{thm*}{Theorem}
\newtheorem{lem}[thm]{Lemma}
\newtheorem{prop}[thm]{Proposition}
\newtheorem{cor}[thm]{Corollary}

\theoremstyle{definition}

\newtheorem{notn}[thm]{Notation}
\newtheorem*{ack}{Acknowledgements}

\theoremstyle{remark}
\newtheorem{rmk}[thm]{Remark}
\newtheorem{eg}[thm]{Example}
\newtheorem{qn}[thm]{Question}

\begin{document}
\title{The Cohen-Macaulay property of affine semigroup rings in dimension 2}
\author{Tony Se}
\email{tonyse@ku.edu}
\author{Grant Serio}
\email{jgserio@ku.edu}
\address{Department of Mathematics\\
University of Kansas\\
 Lawrence, KS 66045-7523 USA}

\date{\today}
\keywords{Cohen-Macaulay, Monomial, Projective Monomial Curve}

\subjclass[2010]{Primary 13H10, 14M05}

\numberwithin{thm}{section}

\begin{abstract}
  Let $k$ be a field and $x,y$ indeterminates over $k$. Let $R=k[x^a,x^{p_1}y^{s_1},\ldots,x^{p_t}y^{s_t},y^b]
  \subseteq k[x,y]$. We calculate the Hilbert polynomial of $(x^a,y^b)$. The multiplicity of this ideal provides part of a criterion for the ring to be Cohen-Macaulay. Next, we prove a simple
  numerical criterion for $R$ to be Cohen-Macaulay in the case when $t=2$. We also provide a simple algorithm which
  identifies the monomial $k$-basis of $R/(x^a,y^b)$. Finally, these simple results are specialized to the case of
  projective monomial curves in $\bbp^3$.
\end{abstract}

\maketitle

\section{Introduction}
\label{sec:intro}

We investigate a criterion for determining whether monomial rings of the form $$R=k[x^a,x^{p_1}y^{s_1},\ldots,x^{p_t}y^{s_t},y^b]$$ are Cohen-Macaulay (or CM). An important special case of this problem is the projective monomial curve: $k[x^n,x^{n-a_1}y^{a_1},\ldots,x^{n-a_t}y^{a_t},y^n]$ for integers $0<a_1<\ldots<a_t<n$. The study of such rings is inspired by the original example of Macaulay \cite[p.\ 98]{Mac}, $k[x^4,x^3y,xy^3,y^4]$. This is a domain with system of parameters $(x^4,y^4)$ in which $\lambda(R/(x^4,y^4))=5\neq e((x^4,y^4))=4$. It is observed that $\dim(R)=2$ and depth$(R)=1$.

One essential manner of viewing this problem is by considering the semigroup of monomials which lie in the ring. For a polynomial ring with $n$ variables, we let each monomial $m$ be a point in $\mathbb{Z}^n$ corresponding to its exponent vector $\log(m)$. These points form a semigroup inside $\mathbb{Z}^n$ whose generators correspond to the monomials generating $R$ over $k$.

One of the most important breakthroughs in the study of the CM property of these rings was made by Hochster.

\begin{thm}[{\cite[Theorem 1]{Ho}}] If $M$ is a monomial semigroup in the variables $x_1,\ldots,x_n$ and $k[M]\subset k[x_1,\ldots,x_n]$ is normal, then $R[M]\subset R[x_1,\ldots,x_n]$ is Cohen-Macaulay for any Cohen-Macualay ring $R$.\end{thm}

While this settles a great number of cases, there are plenty of monomial semigroups for which $k[M]$ is not normal. In particular, the projective monomial curves described in the first paragraph are never normal unless $t=n-1$. To see this, note that we may assume $\gcd(a_1,\ldots,a_t,n)=1$, so that $\frac{x}{y}$ is in the fraction field of $R$. Hence each monomial of the form $x^iy^{n-i}$ is in the fraction field of $R$. In a similar way, many rings of the form $k[x^a,x^{p_1}y^{s_1},\ldots,x^{p_t}y^{s_t},y^b]$ are not normal. For example, the fraction field of $R=k[x^2,x^{11}y,xy^{11},y^3]$ contains $\frac{x^{11}y}{(x^2)^5} = xy$ and $(xy)^6 \in R$.

In the case of simplicial affine semigroups, Goto, Suzuki and Watanbe give another criterion by which to evaluate CM. A semigroup is affine if it may be embedded in $\mathbb{Z}^n$ for some $n$. For any affine semigroup, we may consider the cone of $S$: $C(S) =\{\alpha \in \mathbb{R}^n| k\cdot \alpha \in S \text{ for some } 0\leq k\in \mathbb{R}\}$. The semigroup is said to be simplicial if the cone may be generated in $\mathbb{R}^n$ by $rank(S)$-many linearly independent elements of $S$ as $\mathbb{R}^n$ vectors.

\begin{thm}[{\cite[Theorem 5.1]{GoSuWa}; \cite[Theorem 6.4]{St}}] \label{CMcondition}Let $S$ be a simplicial affine semigroup. Let $e_1,\ldots,e_s$ be elements which span $C_S$. Then $k[S]$ is CM if and only if $$\{ x\in G | x+e_i \in S \text{ and } x+e_j \in S \text{ for some } i\neq j\} = S$$\end{thm}

Goto and Watanabe defined a similar extension $S'$ for a general affine semigroup $S$. Trung and Hoa \cite[Theorem 4.1]{TrHo} identify a topological criterion which, together with $S=S'$, is necessary and sufficient for CM.

The semigroup defined by elements of $R=k[x^a,x^{p_1}y^{s_1},\ldots,x^{p_t}y^{s_t},y^b]$ is an affine semigroup by log(-). Furthermore, any element of $\mathbb{R}^2$ with nonnegative entries may be written as a combination of $(a,0)$ and $(0,b)$. This includes every element of $S$, so $S$ is a simplicial affine semigroup.

The criterion of theorem \ref{CMcondition} is straightforward to check for a single ring, but does not lend itself to analyzing classes of rings. Reid and Roberts \cite{RoRe7} introduce a related notion of a maximal projective monomial curve in order to demonstrate a large class of CM curves. The special case of projective monomial curves continues to be studied.

In this paper, we consider affine semigroup rings in dimension 2. The framework we find helpful emphasizes the congruence classes of the exponent vectors. This allows us to calculate the Hilbert polynomial of $(x^a,y^b)$ in section \ref{sec:multiplicity}. Note that the constants $t_{p,q},s_{p,q},n_{p,q}$ in the following theorem are described in \ref{lem:constant} and \ref{prop:fullbn}. 

\begin{thm*}[\ref{thm:subgp}]Let $R=k[x^a, x^{p_1}y^{q_1},x^{p_2}y^{q_2},\ldots,x^{p_t}y^{q_t},y^b]$. The Hilbert polynomial of $(x^a,y^b)$ is $P(n)=|H|(n+1)+\sum_{(p,q)\in H} (t_{p,q}+s_{p,q})$  where $H\subset (\mathbb{Z}/a\mathbb{Z} \oplus \mathbb{Z}/b\mathbb{Z})$ is the subgroup generated by $(p_1,q_1),(p_2,q_2),\ldots,(p_t,q_t)$. In particular, $e((x^a,y^b)) = |H|$, and the Hilbert function equals the Hilbert polynomial for $n\geq \max_{(p,q)\in H}(n_{p,q})$. \end{thm*}

These details allow us to manipulate the construction of a ring to achieve specific coefficients and level of stabilization for the Hilbert polynomial.

Our calculations also highlight an interesting set of $k$-independent monomials over $R/(x^a,y^b)$. In section \ref{sec:fourgen}, we give more specific calculations for rings with 4 monomial generators. Theorem \ref{thm:fourgen} provides a simple criterion for determining the CM property and Theorem \ref{thm:basis} presents an algorithm that identifies a monomial $k$-basis for $R/(x^a,y^b)$. Section \ref{sec:p3} highlights the application of this work to projective monomial curves in $\mathbb{P}^3$. The CM condition for such rings is as follows.

\begin{thm*}[\ref{thm:p3}] Let $R=k[x^n,x^{n-\ell}y^\ell,x^{n-m}y^m,y^n]$ with $0<\ell<m<n$. Let $b$ be the smallest integer such that there exist integers $a \geq 0$ and $c>0$ with $bm=a\ell+cn$. Choose $a$ such that $n/\!\gcd(n,\ell)>a$. Then $R$ is CM if and only if $b\geq a+c$.\end{thm*}

There are several previous results which bear some resemblence to our conclusions. \cite{Li} provides an algorithm identifying `basis points' in an affine semigroup. These are not the same as our basis elements of $R/(x^a,y^b)$ and the results assume seminormality of the semigroup. Among a study of the defining ideals of monomial rings, \cite[Remark 2.17]{LiPaRo} provides a geometric condition on the monomial basis for a projective monomial curve in $\bbp^3$ to have a CM ring. In contrast, Theorem \ref{thm:p3} is a clearly numerical condition on the exponents. In \cite{MoPaTa} we find another numerical criterion for a projective monomial curve in $\bbp^3$ to be CM, whereas our results in Section \ref{sec:fourgen} do not require the monomial generators of $R$ to have the same degree.

\begin{ack}The results in this paper contribute to the authors' doctoral theses. The authors would like to thank their advisors, Hailong Dao and Daniel Katz. The authors are also indebted to Arindam Banerjee for discussions on this topic.\end{ack}

\section{Asymptotic Behavior of the System of Parameters}
\label{sec:multiplicity}

One important window into the CM property is the asymptotic lengths $R/(x^a,y^b)^n$ for a system of parameters $(x^a,y^b)\in R$. The main result of this section, Theorem \ref{thm:subgp}, gives the Hilbert polynomial for $(x^a,y^b)$ in the ring $R=k[x^a, x^{p_1}y^{q_1},x^{p_2}y^{q_2},\ldots,x^{p_t}y^{q_t},y^b]$. We demonstrate a method by which rings with a specific Hilbert polynomial may be constructed, then return to the implications for evaluating the CM property. A classic characterization of Cohen-Macaulay local rings links this property to multiplicity and length. The following theorem is taken from the text by Matsumura.
\begin{thm}[{\cite[Theorem 17.11]{Ma}}] \label{thm:matsu}
Let $(R,\mathfrak{m})$ be a local ring. The following are equivalent:
\begin{enumerate}[label=(\roman*),align=left,leftmargin=*,nosep]
  \item $R$ is a Cohen-Macaulay ring.
  \item $\lambda(R/I) = e(I)$ for every $I$ generated by a system of parameters.
  \item $\lambda(R/I)=e(I)$ for some $I$ generated by a system of parameters.
\end{enumerate}
\end{thm}

In the following discussion, we assign $x$ to have weight $b$ and $y$ to have weight $a$, so that $\deg(x^\alpha y^\beta) =b\alpha+a\beta$. This is so that $x^a$ and $y^b$ will have equal degree. When we wish to specifically highlight the exponent vector, we will use $\log$: $\log(x^\alpha y^\beta)=(\alpha,\beta)\in \bbz^2$.

\begin{rmk}\label{rmk:a=b}It may be noted that there exists a ring isomorphism $\phi: R\rightarrow R'$ with $\phi(x)=x^b$, $\phi(y)=y^a$ and $R'=k[x^{ab},x^{bp_1}y^{aq_1},\ldots,y^{ab}]$. Without loss of generality, we might have assumed that $a=b$. On the other hand, for $a\neq b$ we may freely assume $\gcd(a,p_1,p_2,\ldots,p_t)=\gcd(b,q_1,\ldots,q_t)=1$.\end{rmk}

\begin{notn}\label{notn}We consider $(x^a,y^b)$ and its powers and find it convenient to denote $X:= x^a$ and $Y:=y^b$. We sort monomials in $R$ into congruence classes based on their exponents in $H\subset \qtgp{a}{b}$. For each $(p,q)\in H$, we choose $\alpha_{p,q}$ to be one monomial of minimal (weighted) degree in $R\setminus (X,Y)$. We will approximate $\lambda\left( \frac{(X,Y)^n}{(X,Y)^{n+1}}\right)$ by the size of  $$A_n:= \cup_{(p,q)\in H} \{\alpha_{p,q}X^n, \alpha_{p,q}X^{n-1}Y,\ldots,\alpha_{p,q}Y^n\}$$
The remaining monomials of $(X,Y)^n$ outside both $(X,Y)^{n+1}$ and $A_n$ will be denoted $B_n$.

Set $\alpha_{p,q} = x^{\ell a+p}y^{mb+q}= {}_{0}\beta=\beta_0$. For $i<n$, let $_i\beta := x^{(\ell-i)a+p}y^{(m+j)b+q}$ with $j$ the least possible integer such that $_i\beta \in R$. It may be that there is no such monomial, in which case we do not consider $_i\beta$ to be defined. Similarly, $\beta_i := x^{(\ell+j)a+p}y^{(m-i)b+q}$.\end{notn}

\begin{lem}\label{lem:higherbeta}Let $i>j$. If $\deg(\beta_i)\leq \deg(\beta_h)$ for all $h>j$ such that $\beta_h$ is well-defined, then $\beta_i X^{n-i+h}Y^{i-h}\notin (X,Y)^{n+1}$ for $i\geq h>j$.
\end{lem}

\begin{proof}If $\beta_i X^{n-i+h}Y^{i-h}\in (X,Y)^{n+1}$, then there exists $\gamma \in R$ with $\gamma X^{n+1-c}Y^c = \beta_i X^{n-i+h}Y^{i-h}$ for some $c\in \bbn$. By the minimality of the $\beta$'s, $\beta_{h+c}$ divides $\gamma$, so that $\deg(\beta_{h+c})<\deg(\beta_i)$.
\end{proof}

\begin{lem}\label{lem:constant}Fix $(p,q)\in H$. Let $s_{p,q} = s$ be the highest integer such that $\beta_s$ is defined. Let $u$ be the maximum value of $i-j-1$, $s\geq i>j\geq 0$ such that $\deg(\beta_i)<\deg(\beta_h)$ for all $i>h>j$. Let $\mathcal{U}_{p,q,n}=\mathcal{U}_n =\{\beta_i X^{n-c}Y^c| 0<i\leq s, 0\leq c\leq n\}$. Then $|B_n\cap \mathcal{U}_n| \leq s$ with equality if and only if $n\geq u$.
\end{lem}

\begin{proof}We first claim $|B_n\cap\mathcal{U}_n|\leq s$. $B_n\subset (X,Y)^n\setminus (X,Y)^{n+1}$ which means at most one monomial in $\mathcal{U}_n$ of a given $y$-exponent may lie in $B_n$. If $c\geq i$, then $\alpha_{p,q}X^{n-c+i}Y^{c-i}$ divides $\beta_i X^{n-c}Y^c$, so $\beta_i X^{n-c}Y^c$ cannot be in $B_n$. There are only $s$-many other $y$-exponents the monomials in $\mathcal{U}_n$ might take.

Let $s=i_v>i_{v-1}>\ldots>i_0=0$ be integers such that $\deg(\beta_{i_v})\geq \deg(\beta_{i_{v-1}})\geq\ldots\geq \deg(\beta_0)$ and $\deg(\beta_{i_z})< \deg(\beta_h)$ for all $i_z>h>i_{z-1}$.

Note that each pair $i_z,i_{z-1}$ satisfies the defining condition of $u$. Moreover, since $\deg(\beta_{i_z})\leq \deg(\beta_h)$ for any $h>i_z$, any pair $i,j$ with $i>i_z>j$ fails the condition that defines $u$. Hence $u=\max_z(i_z-i_{z-1}-1)$.

Suppose $n\geq u$ and consider the following monomials:
\begin{align*}&\beta_s X^n, \beta_s X^{n-1}Y,\ldots, \beta_sX^{n-(s-i_{v-1}-1)}Y^{s-i_{v-1}-1},\\
&\beta_{i_{v-1}}X^n, \ldots, \beta_{i_{v-1}}X^{n-(i_{v-1}-i_{v-2}-1)}Y^{i_{v-1}-i_{v-2}-1},\\
&\ldots,\\ &\beta_{i_1}X^n,\ldots, \beta_{i_1} X^{n-(i_1-1)}Y^{i_1-1}\end{align*}
By Lemma \ref{lem:higherbeta}, each of these monomials lies outside $(X,Y)^{n+1}$, so $|B_n \cap \mathcal{U}_n| = s$.

Suppose $n<u$ and let $i,j$ be indices satisfying $\deg(\beta_i)<\deg(\beta_h)$ for all $i>h>j$ and $i-j-1=u$. Consider the set of monomials $\beta_h X^{i-h-1}Y^{h-i+n+1}$ with $i>h\geq i-n-1$. By the assumption on $n$, $\deg(\beta_i)<\deg(\beta_h)$ for $i>h\geq i-n-1$ so $\beta_i Y^{n+1}$ divides each $\beta_h X^{i-h-1}Y^{h-i+n+1}$. Moreover, these $\beta_h X^{i-h-1}Y^{h-i+n+1}$ are the only monomials in $\mathcal{U}_n$ with the same $y$-exponent as $\beta_iY^{n+1}$. Hence $B_n\cap\mathcal{U}_n$ does not contain a monomial with this $y$-exponent and $|B_n\cap\mathcal{U}_n|<s$.
\end{proof}

Symmetry allows us to apply this result to $_t\beta,\ldots,{}_0\beta$. Define $u'$ for $_t\beta,\ldots,{}_0\beta$ as $u$ is defined for $\beta_s,\ldots,\beta_0$. For $n\geq n_{p,q}:=\max(u,u')$, this yields $|B_n \cap \{\gamma | \log(\gamma) \equiv (p,q)\in H\}| = s+t$.

\begin{prop}\label{prop:fullbn} $|B_n|\leq \sum_{(p,q)\in H} (t_{p,q}+s_{p,q})$ with equality if and only if $n \geq \max_{(p,q)\in H} (n_{p,q})$.\end{prop}

\begin{proof}$B_n$ is the disjoint union of $\mathcal{U}_{p,q,n}\cap B_n$ as $(p,q)$ varies in $H$. By \ref{lem:constant}, $|\cup_{(p,q)\in H} (\mathcal{U}_{p,q,n}\cap B_n)| = \sum_{(p,q)\in H} |\mathcal{U}_{p,q,n}\cap B_n| = \sum_{(p,q)\in H} (t_{p,q}+s_{p,q})$ if and only if $n\geq \max_{(p,q)\in H} n_{p,q}$.\end{proof}

Not only is $|B_n|$ constant for large values of $n$, it determines the constant of the Hilbert polynomial $P(n) = \lambda((X,Y)^n/(X,Y)^{n+1})$. We demonstrate this by calculating the multiplicity from the growth of $|A_n|$.

\begin{thm} \label{thm:subgp}
Let $R=k[x^a, x^{p_1}y^{q_1},x^{p_2}y^{q_2},\ldots,x^{p_t}y^{q_t},y^b]$. The Hilbert polynomial of $(x^a,y^b)$ is $P(n)=|H|(n+1)+\sum_{(p,q)\in H} (t_{p,q}+s_{p,q})$  where $H\subset (\mathbb{Z}/a\mathbb{Z} \oplus \mathbb{Z}/b\mathbb{Z})$ is the subgroup generated by $(p_1,q_1),(p_2,q_2),\ldots,(p_t,q_t)$. In particular, $e((x^a,y^b)) = |H|$, and the Hilbert function equals the Hilbert polynomial for $n\geq \max_{(p,q)\in H}(n_{p,q})$.\end{thm}

\begin{proof}$(X,Y)^n$ is generated by the $n+1$ monomials $X^n,X^{n-1}Y,\ldots,Y^n$.
$R$ has dimension 2, so $\lambda\left( \frac{(X,Y)^n}{(X,Y)^{n+1}}\right)$ is given by a linear polynomial, $P(n)$, for sufficiently high $n$. Applying the notation from \ref{notn}, we have $P(n)=|A_n|+|B_n|$. But $A_n = \cup_{(p,q) \in H} \{\alpha_{p,q}X^{n-i}Y^i|0\leq i\leq n\}$, so $|A_n|= |H|(n+1)$. Together with Lemma \ref{lem:constant}, we have $P(n)=|A_n|+|B_n| = |H|(n+1)+\sum_{(p,q)\in H} (t_{p,q}+s_{p,q})$.
\end{proof}

Taken together, these results allow us to construct rings with arbitrary conditions on the Hilbert polynomial and the level of its stabilization.

\begin{cor}\label{lem:anyHC}
Given any subgroup $0 \neq H\subset \qtgp{a}{b}$ and integers $C,m\geq 0$, there exists $R$
such that $(x^a,y^b)$ has Hilbert polynomial $P(n)=|H|(n+1)+C$ which equals the Hilbert function exactly for $n\geq m$.
\end{cor}

\begin{proof}
  Fix $(0,0) \neq (p_0,q_0) \in H$ and an integer $N \geq C+m+1$. For any $(0,0)\neq (p,q) \in H$ we
  let $\alpha_{p,q} = x^p y^q X^N Y^{N+C+m}$. Let ${\beta_{p_0,q_0,j}} = x^{p_0} y^{q_0}
  X^{N+j} Y^{N+C+m-j}$ for $j=0,1,2,\dots,C-1,C+m$. Let $S' = \{ \alpha_{p,q} \mid (0,0) \neq
  (p,q) \in H\} \cup \{ \beta_{p_0,q_0,j} \mid j=0,1,2,\dots,C-1,C+m\}$, $S = S' \cup \{x^a,y^b\}$
  and $R = k[S]$. We will show that $(x^a,y^b) \subset R$ has the required Hilbert polynomial.
  
  Let $\gamma,\delta \in S'$, and let $\log_x, \log_y$ denote the respective exponents of a monomial. 
  Then $\log_x(\gamma\delta) \geq 2Na \geq (N+C+m+1)a > (N+C+m)a+p$
  and $\log_y(\gamma\delta) \geq 2Nb > (N+C+m)b + q$ for any $0\leq p < a$ and $0\leq q < b$.
  In particular, $\log_x(\gamma\delta) > \log_x(\alpha_{p,q})$ and $\log_y(\gamma\delta)
  > \log_y(\alpha_{p,q})$ for any $\alpha_{p,q}$, so there is no $_{p,q,i}\beta$ for any $(p,q)
  \in H$ or $\beta_{p,q,j}$ for any $(p,q) \neq (p_0,q_0)$ or for $(p,q) = (p_0,q_0)$ with
  $j > C+m$. Similarly, $\log_x(\gamma\delta) > \log_x(\beta_{p_0,q_0,j})$ and
  $\log_y(\gamma\delta) > \log_y(\beta_{p_0,q_0,j})$ for any $j=0,1,2,\dots,C-1,C+m$, so
  $\deg(\beta_{p_0,q_0,j}) > \deg(\beta_{p_0,q_0,C+m})$ for all $j=C,C+1,\dots,C+m-1$. In particular, $\beta_{p_0,q_0,j} = \beta_{p_0,q_0,C+m}Y^{C+m-j}$.
  This gives the maximum $u$ as in Lemma~\ref{lem:constant} as $(C+m)- (C-1)-1 =m$. Therefore
  $(x^a,y^b) \subset R$ has Hilbert polynomial $P(n)=|H|(n+1)+C$ which equals the Hilbert function
  exactly at $n\geq m$ by Proposition~\ref{prop:fullbn}.
\end{proof}

Let us return to consideration of the CM property. In general, $\lambda(R/(X,Y))\geq e((X,Y))$ and equality implies CM by \ref{thm:matsu}.

\begin{prop}\label{prop:empty}The following are equivalent:\begin{enumerate}[label=(\roman*),align=left,leftmargin=*,nosep] \item $R$ is CM \item $B_n = \emptyset$ for all $n$. \item $B_i=\emptyset$ for some $i$.\end{enumerate}\end{prop}

\begin{proof}$(i)\Rightarrow (ii)$	If $R$ is CM, then by Theorems \ref{thm:matsu} and \ref{thm:subgp}, $\lambda (R/(X,Y)) = |H|$. Then every monomial of $R$ may be written as $\alpha_{p,q}X^{i}Y^j$ for some $i,j\in \bbn$. Hence $B_n=\emptyset$ for all $n$.

$(iii)\Rightarrow (i)$	If $R$ is not CM, then $\lambda(R/(X,Y)) >|H|$. By the pigeonhole principle, some congruence class $(p,q)\in H$ must be associated with two monomials in $R\setminus(X,Y)$. That is, for some $(p,q)\in H$, there is $_i\beta$ or $\beta_j$. If $s>0$ is the highest integer such that $\beta_s$ is defined, then $\beta_sX^i\in B_i$ for all $i$.\end{proof}

An alternative manner of viewing this result helps to motivate the calculations in the following sections. Impose the reverse lexicographic order on monomials in $R$. Let $\mu_{p,q}$ be the least monomial in this order such that $\log(\mu_{p,q})\equiv (p,q)$. We use $\basis_0$ to indicate the collection of $\mu_{p,q}$ for all $(p,q)\in H$. Alternatively, we may use the lexicographic order of the monomials, and form a set $\basis_0'$ of elements $\mu_{p,q}'$.

\begin{prop}The following are equivalent:\begin{enumerate}[label=(\roman*),align=left,leftmargin=*,nosep] \item $R$ is CM \item $\basis_0=\basis_0'$ \item $\mu_{p,q}=\mu_{p,q}'$ for all $(p,q)\in H$.\end{enumerate}\end{prop}

\begin{proof}$(i)\Rightarrow (ii)$: If $R$ is CM, then $\lambda(R/(x^a,y^b)) = e((x^a,y^b))=|\basis_0|=|\basis_0'|$ by Theorem \ref{thm:subgp}. Elements of $\basis_0$ are outside $(X,Y)$ by construction, so $\basis_0$ is a $k$-basis for $R/(X,Y)$. But the same is true for $\basis_0'$ and there is only one $k$-basis consisting of monomials.

$(ii)\Rightarrow (iii)$: Suppose $\basis_0=\basis_0'$. For each congruence class in $H$, $\basis_0$ and $\basis_0'$ contain exactly one element whose log lies in that class. Since $\mu_{p,q} \in \basis_0'$, it must be that $\mu_{p,q}=\mu_{p,q}'$.

$(iii)\Rightarrow (i)$: Suppose $\mu_{p,q}=\mu_{p,q}'$, so that $\mu_{p,q}$ has the smallest $x$-exponent and the smallest $y$-exponent of any monomial in its congruence class. $B_n \cap \{\beta| \log(\beta)=(p,q)\} = \emptyset$. Since this holds for all $(p,q)\in H$, $B_n= \emptyset$ and $R$ is CM by Proposition \ref{prop:empty}.\end{proof}

\section{Semigroup rings with four generators}
\label{sec:fourgen}

In this section, we will consider semigroup rings of the form $R=k[x^d,x^ey^{\ell},x^fy^m,y^n]$
with $d,n>0$, $e,f,\ell,m \geq 0$ and $(e,\ell) \neq (0,0)$.
Our first main result in this section is Theorem~\ref{thm:fourgen}, which gives a simple criterion to determine
whether $R$ is Cohen-Macaulay. The second main result is Theorem~\ref{thm:basis}, which gives an algorithm
to generate a $k$-basis of $R/(x^d,y^n)$.
As noted in Remark~\ref{rmk:a=b}, one may assume that $d=n$ for most results in this section, whereas
Corollary~\ref{cor:length} is probably best stated without assuming $d=n$.

\begin{notn}\label{notn:abch}
  Given a group $G$ and an element $g \in G$, we write $\ord(g,G)$ to denote the order of $g$ in $G$.
  For elements $(g,h),(g',h') \in \bbz^2$ we let $(g,h) \prec (g',h')$ if and only if
  either $h' > h$ or $h'=h$ but $g' > g$.
  \separate

  Throughout this section, we fix $a_i,b_i \in \bbn$ and $(g_i,h_i) \in d\bbz \oplus n\bbz$, $i=1,2,3$
  as follows. Let $(g_1,h_1) \in d\bbz \oplus n\bbz$ be the smallest element with respect to $\prec$
  such that there are positive integers $a_1,b_1$ with $b_2 \geq b_1$ ($b_2$ to be defined below) and
  \begin{equation} \label{eqn:(d,n)}
    a_1(e,\ell) + b_1(f,m) = (g_1,h_1)
  \end{equation}
  
  Let $b_2$ be the smallest positive integer such that there exist $a_2 \geq 0$ and
  $(g_2,h_2) \in d\bbz \oplus n\bbz$ with either at least one of $g_2,h_2$ being positive or
  $(g_2,h_2)=(0,0)$ such that
  \begin{equation} \label{eqn:(f,m)}
    -a_2(e,\ell) + b_2(f,m) = (g_2,h_2)
  \end{equation}
  We choose $a_2 < \ord((e,\ell),(\bbz/d\bbz) \oplus (\bbz/n\bbz))$.
  
  Let $a_3$ be the smallest positive integer such that there exist $b_3 \geq 0$ and
  $(g_3,h_3) \in d\bbz \oplus n\bbz$ with $g_3,h_3\geq 0$ and $g_3,h_3$ not both 0 such that
  \begin{equation} \label{eqn:(e,l)}
    a_3(e,\ell) - b_3(f,m) = (g_3,h_3)
  \end{equation}
  We choose $b_3 < \ord((f,m),(\bbz/d\bbz) \oplus (\bbz/n\bbz))$.
\end{notn}

\begin{lem}
  We have $a_3 > a_2$ and $b_2 > b_3$.
\end{lem}

\begin{proof}
   If $a_3 \leq a_2$, then $\eqref{eqn:(f,m)} + \eqref{eqn:(e,l)}$
  gives
  \[
    (b_2-b_3)(f,m) = (g_2+g_3,h_2+h_3) + (a_2-a_3)(e,\ell)
  \]
  By the definitions of $b_2$ and $a_3$, at least one of $g_2+g_3$ or $h_2+h_3$ is positive, so $b_2-b_3
  >0$. If $b_3=0$, then the definition of $a_3$ gives $\ord((e,\ell),(\bbz/d\bbz) \oplus (\bbz/n\bbz))=
  a_3 \leq a_2$, contradicting the choice of $a_2$, so $b_3>0$. But then $b_2>b_2-b_3>0$ contradicts
  the minimality of $b_2$ in \eqref{eqn:(f,m)}. Therefore $a_3 > a_2$.
  \separate

  If $b_2 \leq b_3$, then $\eqref{eqn:(f,m)} + \eqref{eqn:(e,l)}$ gives
  \[
    (a_3-a_2)(e,\ell) = (g_2+g_3,h_2+h_3) + (b_3-b_2)(f,m)
  \]
  Again at least one of $g_2+g_3$ or $h_2+h_3$ is positive, so $a_3-a_2>0$. If $a_2=0$, then the
  definition of $b_2$ gives $\ord((f,m),(\bbz/d\bbz)\oplus(\bbz/n\bbz)) = b_2 \leq b_3$, contradicting
  the choice of $b_3$, so $a_2>0$. If $g_2+g_3$ and $h_2+h_3$ are both nonnegative,
  then $a_3>a_3-a_2>0$ contradicts the minimality of $a_3$ in \eqref{eqn:(e,l)}. So without loss of
  generality, suppose that $h_2+h_3>0$ and $g_2+g_3 < 0$, so $g_2<0$ and $h_2>0$.
  Let $q \in \bbz$, $q>1$ be such that $a_3-(q-1)a_2>0$ but $a_3-qa_2 \leq 0$. Then
  $(q-1)\eqref{eqn:(f,m)} + \eqref{eqn:(e,l)}$ gives
  \[
    (a_3-(q-1)a_2)(e,\ell) = ((q-1)g_2+g_3,(q-1)h_2+h_3) + (b_3-(q-1)b_2)(f,m)
  \]
  Then $(q-1)g_2+g_3<0$ gives $b_3-(q-1)b_2>0$. Next, $q\eqref{eqn:(f,m)} + \eqref{eqn:(e,l)}$ gives
  \[
    (qb_2-b_3)(f,m) = (qg_2+g_3,qh_2+h_3) + (qa_2-a_3)(e,\ell)
  \]
  Since $qh_2+h_3>0$ and $qa_2-a_3 \geq 0$ we have $qb_2-b_3>0$. But then $qb_2-b_3=
  b_2-((b_3-(q-1)b_2)<b_2$ contradicts the minimality of $b_2$ in \eqref{eqn:(f,m)}.
  Therefore $b_2>b_3$.
\end{proof}

\begin{lem} \label{lem:notcong}
  If $u,v \in \bbz$ are such that $a_3 > u \geq 0$, $b_2 > v \geq 0$ and
  $(u,v) \neq (0,0)$, then $u(e,\ell) - v(f,m) \notin d\bbz \oplus n\bbz$.
\end{lem}

\begin{proof}
  Suppose that $u(e,\ell) =  (g,h) +v(f,m)$ for some $(g,h) \in d\bbz \oplus n\bbz$.
  If $g,h \geq 0$ and $g,h$ are not both 0, then $u>0$, contradicting the minimality of $a_3$.
  Otherwise we have $v(f,m) = (-g,-h) + u(e,\ell)$ with $-g>0$, $-h>0$ or
  $(-g,-h)=(0,0)$, contradicting the minimality of $b_2$. Therefore such $(g,h)$ does not exist.
\end{proof}

\begin{lem} \label{lem:123}
  Suppose that $a,b \in \bbn$ and $(g,h) \in d\bbz \oplus n\bbz$ are such that
  \begin{equation} \label{eqn:notmin}
    a(e,\ell) + b(f,m) = (g,h)
  \end{equation}
  \begin{enumerate}[label=(\roman*),align=left,leftmargin=*,nosep]
    \item \label{item:onaxis} If $a_2=a=0$, then $b_2 \mid b$. If $b_3=b=0$, then $a_3 \mid a$.
    \item \label{item:bgtb2} If $a \geq a_3$ and $b>b_2$, then $g \geq g_2+g_3$ and $h \geq h_2+h_3$. If
    $a>a_3-a_2$ or $(f,m) \neq (0,0)$, then $(g,h) \neq (g_2+g_3,h_2+h_3)$.
    \item \label{item:bleqb2} Suppose that $a \geq a_3$, $b \leq b_2$ and $(a,b) \neq (a_3-a_2,b_2-b_3)$. If
    either $b \geq b_2-b_3$ or $b_2-b_3>b$ and $b_3,b$ are not both 0, then
    $g \geq g_2+g_3$, $h \geq h_2+h_3$ and $(g,h) \neq (g_2+g_3,h_2+h_3)$.
    \item \label{item:inside} If $a \leq a_3$, $b \leq b_2$, $a_2,a$ are not both 0 and $b_3,b$ are not both 0, then
  $(a,b) = (a_3-a_2,b_2-b_3)$ or $(0,0)$, or $a>a_3-a_2$ and $b>b_2-b_3$. In fact, there exists
  $q \in \bbn$ such that $a=q(a_3-a_2)$ and $b=q(b_2-b_3)$.
  \end{enumerate}
  In particular, we have $\eqref{eqn:(d,n)} = \eqref{eqn:(f,m)} + \eqref{eqn:(e,l)}$, and one may use
  either lexicographic or reverse lexicographic ordering in the definition of \eqref{eqn:(d,n)}.
\end{lem}

\begin{proof}
  \ref{item:onaxis}: If $a_2=0$, then $b_2 = \ord((f,m),\qtgp{d}{n})$, so if $a=0$, then $b_2 \mid b$. Similarly, if
  $b_3=b=0$, then $a_3 = \ord((e,\ell),\qtgp{d}{n}) \mid a$.
  \separate
  
  \ref{item:bgtb2} and \ref{item:bleqb2}: Now since $a_3>a_2$ and $b_2>b_3$, $\eqref{eqn:(f,m)} + \eqref{eqn:(e,l)}$ gives
  \begin{equation} \label{eqn:mincand}
    (a_3-a_2)(e,\ell) + (b_2-b_3)(f,m) = (g_2+g_3,h_2+h_3)
  \end{equation}
  Suppose that $a\geq a_3 \geq a_3-a_2$ and $b \geq b_2-b_3$. Then $\eqref{eqn:notmin}-
  \eqref{eqn:mincand}$ gives $g \geq g_2+g_3$ and $h \geq h_2+h_3$. Suppose furthermore that
  $(a,b) \neq (a_3-a_2,b_2-b_3)$. If $a>a_3-a_2$, then $g>g_2+g_3$ or $h>h_2+h_3$.
  If $a=a_3-a_2$ so that $b>b_2-b_3$, then $b_2 \leq \ord((f,m),\qtgp{d}{n}) \mid b-(b_2-b_3)$, so
  $b>b_2$. So suppose that $a \geq a_3$, $b < b_2-b_3$
  and $b_3,b$ are not both 0. If $g<g_2+g_3$ or $h<h_2+h_3$ or $(g,h)=(g_2+g_3,h_2+h_3)$,
  then $\eqref{eqn:mincand}-\eqref{eqn:notmin}$ gives
  \[
    (b_2-b_3-b)(f,m) = (g_2+g_3-g,h_2+h_3-h) + (a-a_3+a_2)(e,\ell),
  \]
  contradicting the minimality of $b_2$. So $g\geq g_2+g_3$, $h\geq h_2+h_3$ and
  $(g,h) \neq (g_2+g_3,h_2+h_3)$.
  \separate

  \ref{item:inside}: Now suppose that $a_3\geq a$, $b_2\geq b$, $a_2,a$ are not both 0, $b_3,b$ are not both 0,
  $(a,b) \neq (a_3-a_2,b_2-b_3)$ and $(a,b) \neq (0,0)$. Then by Lemma~\ref{lem:notcong}, we cannot
  have $a\geq a_3-a_2$ and $b \leq b_2-b_3$, or $a \leq a_3-a_2$ and $b \geq b_2-b_3$.
  So suppose that $a<a_3-a_2$ and $b<b_2-b_3$. If $a \neq 0$, then $\eqref{eqn:(e,l)}-
  \eqref{eqn:notmin}$ gives
  \[
    (a_3-a)(e,\ell)-(b+b_3)(f,m)=(g_3-g,h_3-h),
  \]
  contradicting Lemma~\ref{lem:notcong}. Similarly, $b \neq 0$ and $\eqref{eqn:notmin}-
  \eqref{eqn:(f,m)}$ gives a contradiction. Therefore $a>a_3-a_2$ and $b>b_3-b_2$. In such a case,
  let $q \in \bbn$ be such that $a-(q-1)(a_3-a_2),b-(q-1)(b_2-b_3)>0$ but one of $a-q(a_3-a_2)$ or
  $b-q(b_2-b_3)$ is nonpositive. By what we just proved, we have $(a-(q-1)(a_3-a_2),b-(q-1)(b_2-b_3))
  =(a_3-a_2,b_2-b_3)$, so $a=q(a_3-a_2)$ and $b=q(b_2-b_3)$.
\end{proof}

\begin{notn}\label{notn:b0}
  Let $a,b$ denote natural numbers. We let
  \begin{align*}
    B_0 &= \{ (a,b) \mid a < a_1 \text{ and } b < b_2 \}
              \cup \{ (a,b) \mid a < a_3 \text{ and } b < b_1 \}\\
        &= \{ (a,b) \mid a < a_3 \text{ and } b < b_2 \}
              \setminus \{ (a,b) \mid a \geq a_1 = a_3-a_2 \text{ and } b \geq b_1 = b_2-b_3 \}
  \end{align*}
  Let us write $\vct{a}{b} = a(e,\ell)  + b(f,m)$. We write $\vct{a}{b} \equiv \vct{a'}{b'}$ to mean
  $\vct{a}{b} - \vct{a'}{b'} \in \subgp{d}{n}$. We let $H$ be the subgroup of $\qtgp{d}{n}$
  generated by $(e,\ell) = \vct{1}{0}$ and $(f,m) = \vct{0}{1}$.
\end{notn}

\begin{rmk}
  We may visualize the set $B_0$ as follows. For $(a,b) \in \bbn \times \bbn$, the first coordinate
  $a$ increases to the right and the second coordinate $b$ increases downwards.
  
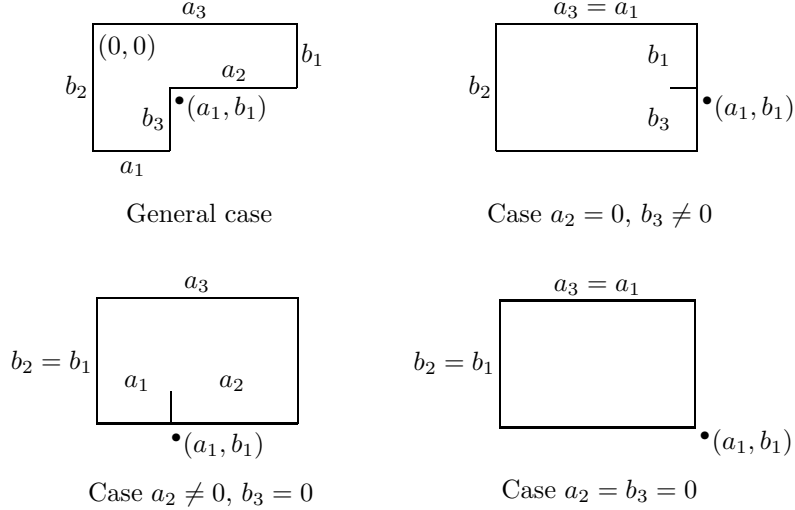
\begin{center}
    \begin{tabular}{r@{} c c @{}l}
      & \multicolumn{2}{c}{$a_3$} & \\ \cline{2-3}
      \multirow{3}{*}{\rule{18pt}{0pt} $b_2$\,}
      & \multicolumn{1}{|@{}l}{\raisebox{3pt}{\,$(0,0)$}}
      & \multicolumn{1}{c|}{\parbox[c][24pt][c]{38pt}{\rule{0pt}{18pt} \hspace{8pt} $a_2$}}
      & $\, b_1$\\ \cline{3-3}
      & \multicolumn{1}{|r@{}|}{\parbox[c][24pt][c]{15pt}{\hfill $b_3\,$}}
      & \multicolumn{1}{@{}l}{\raisebox{4pt}{\,$^{\bullet}(a_1,b_1)$}}\\
      \cline{2-2}
      & \hspace{3pt} $a_1$\\
      & \multicolumn{2}{c}{\rule{0pt}{15pt} General case}
    \end{tabular}
    \qquad
    \begin{tabular}{r@{} r@{} c @{}l}
      & \multicolumn{2}{c}{$a_3=a_1$} & \\ \cline{2-3}
      \multirow{3}{*}{\rule{18pt}{0pt} $b_2$\,}
      & \multicolumn{1}{|r@{}}{\parbox[c][24pt][c]{61pt}{\hfill $b_1$}}
      & \multicolumn{1}{c|}{}\\ \cline{3-3}
      & \multicolumn{1}{|r@{}}{\parbox[c][24pt][c]{61pt}{\hfill $b_3$}}
      & \multicolumn{1}{c|}{}
      & \multicolumn{1}{@{}l}{\raisebox{4pt}{\,$^{\bullet}(a_1,b_1)$}}\\
      \cline{2-3}
      & \hspace{3pt} \phantom{$a_1$}\\
      \multicolumn{4}{c}{\rule{0pt}{15pt} Case $a_2=0$, $b_3 \neq 0$ \rule{3pt}{0pt}}
    \end{tabular}\\[12pt]
    \begin{tabular}{r@{} c @{}l @{}l}
      & \multicolumn{2}{c}{$a_3$}\\ \cline{2-3}
      \multirow{3}{*}{\raisebox{9pt}{$b_2=b_1$}\,}
      & \multicolumn{1}{|c}{\parbox[c][34pt][b]{18pt}{\rule{2pt}{0pt} $a_1$}}
      & \multicolumn{1}{c|}{\parbox[c][34pt][b]{38pt}{\rule{10pt}{0pt} $a_2$}}\\
      & \multicolumn{1}{|c|}{}
      & \multicolumn{1}{|c|}{}
      & \phantom{$\, b_1$}\\
      \cline{2-3}
      && \rule{0pt}{11pt}$^{\bullet}(a_1,b_1)$\\
      \multicolumn{4}{c}{\rule{0pt}{15pt} \rule{18pt}{0pt} Case $a_2 \neq 0$, $b_3=0$}
    \end{tabular}
    \qquad
    \begin{tabular}{r@{} c @{}l}
      & $a_3=a_1$ & \\ \cline{2-2}
      \multirow{3}{*}{$b_2=b_1$\,}
      & \multicolumn{1}{|c|}{\parbox[c][24pt][c]{64pt}{\rule{60pt}{0pt}}}\\
      & \multicolumn{1}{|c|}{\parbox[c][24pt][c]{64pt}{}}\\
      \cline{2-2}
      && \multicolumn{1}{@{}l}{\,$^{\bullet}(a_1,b_1)$}\\
      \multicolumn{3}{c}{\rule{0pt}{15pt} Case $a_2=b_3=0$ \rule{3pt}{0pt}}
    \end{tabular}
  \end{center}
\end{rmk}

\begin{lem} \label{lem:sizecand}
  We have $|B_0| = |H|$.
\end{lem}

\begin{proof}
  $|B_0| \geq |H|$: We will show that for every $\vct{a'}{b'}$ with $a',b' \in \bbn$, there exists
  $(a,b) \in B_0$ such that $\vct{a}{b} \equiv \vct{a'}{b'}$. First, we show that there exist
  $a'',b'' \in \bbn$ such that $\vct{a'}{b'} \equiv \vct{a''}{b''}$ and $b'' < b_2$. Let $q,r \in \bbn$ be
  such that $b' = qb_2 + r$ as in the Euclidean algorithm. Then from \eqref{eqn:(f,m)} we have
  $\vct{0}{b_2} \equiv \vct{a_2}{0}$, so $\vct{a'}{b'} \equiv \vct{a'+qa_2}{r}$ with $b_2>r \geq 0$.
  \separate

  So assume that $b'<b_2$. We will now reduce to the case that $a'<a_3$. It suffices to
  show that if $a' \geq a_3$, then there exist $a'',b'' \in \bbn$ such that $\vct{a'}{b'} \equiv
  \vct{a''}{b''}$, $a''<a'$ and $b''<b_2$.
  \separate

  \textit{Case 1:} $b' \geq b_1$. From \eqref{eqn:(d,n)} we have $\vct{a_1}{b_1} \equiv \vct{0}{0}$,
  so $\vct{a'}{b'} \equiv \vct{a'-a_1}{b'-b_1}$ with $a'>a'-a_1 \geq a'-a_3 \geq 0$ and $b_2>b'>b'-b_1
  \geq 0$.
  \separate

  \textit{Case 2:} $b'<b_1$ and $b'+b_3<b_2$. From \eqref{eqn:(e,l)} we have $\vct{a_3}{0} \equiv
  \vct{0}{b_3}$,  so $\vct{a'}{b'} \equiv \vct{a'-a_3}{b'+b_3}$.
  \separate

  \textit{Case 3:} $b'<b_1$ and $b'+b_3 \geq b_2$. From \eqref{eqn:(e,l)} and \eqref{eqn:(f,m)}
  we have $\vct{a'}{b'} \equiv \vct{a'-a_3+a_2}{b'+b_3-b_2}$ with
  $a'>a'-a_1=a'-a_3+a_2$ and $b_2>b'>b'-b_1=b'+b_3-b_2 \geq 0$.
  \separate

  So suppose that $a'<a_3$ and $b'<b_2$ but $a' \geq a_1$ and $b' \geq b_1$. Let $q \in \bbn$
  be such that $a'-qa_1, b'-qb_1 \geq 0$ but $a'-(q+1)a_1$ or $b'-(q+1)b_1$ is negative, so that
  $a'-qa_1<a_1$ or $b'-qb_1<b_1$. Then $\vct{a'}{b'} \equiv \vct{a'-qa_1}{b'-qb_1}$ and
  $(a'-qa_1,b'-qb_1) \in B_0$.
  \separate

  $|B_0| \leq |H|$: Suppose that $(a,b),(a',b') \in B_0$ and $(a,b) \neq (a',b')$. If $a'-a \geq 0$
  and $b'-b\leq 0$ then $\vct{a}{b} \not\equiv \vct{a'}{b'}$ by Lemma~\ref{lem:notcong}. If $a'-a,b'-b
  \geq 0$ then $\vct{a}{b} \not\equiv \vct{a'}{b'}$ by Lemma~\ref{lem:123}. Therefore $|B_0| \leq |H|$.
\end{proof}

\begin{notn}\label{notn:vectorx}
  Given $a,b \in \bbn$, we define the monomial
  \[
    \vec{x}^{\vct{a}{b}} = (x^ey^{\ell})^a(x^fy^m)^b = x^{ae+bf}y^{a\ell+bm}
  \]
  We also define the set of monomials $\basis_0 = \{ \vec{x}^{\vct{a}{b}} \mid (a,b) \in B_0 \}$.
\end{notn}

\begin{rmk} \label{rmk:quadrant}
  Let $a,b,a',b' \in \bbn$.
  \begin{enumerate}[label=(\roman*),align=left,leftmargin=*,nosep]
    \item If $a' \geq a$, $b' \geq b$ and $\vct{a'}{b'} - \vct{a}{b} = (g,h)$, then $g,h \geq 0$.
    \item Equations \eqref{eqn:(d,n)} and \eqref{eqn:(e,l)} show that $\compvec{x}{a_1}{b_1}
    \in (x^d,y^n)$ and $\compvec{x}{a_3}{0} \in \compvec{x}{0}{b_3} (x^d,y^n)$. Hence
    $\compvec{x}{a'}{b'} \in (x^d,y^n)$ if $a' \geq a_3$, or $a' \geq a_1$ and $b' \geq b_1$.
    \item If $a' \leq a$, $b' \leq b$ and $(a,b) \in B_0$, then $(a',b') \in B_0$.
  \end{enumerate}
\end{rmk}

\begin{lem} \label{lem:indep}
  Given a set $S \subseteq \bbn \times \bbn$, the set of monomials $\{ \compvec{x}{a}{b} \mid (a,b)
  \in S \}$ is linearly independent in $R/(x^d,y^n)$ over $k$ if and only if:
  \begin{enumerate}[label=(\roman*),align=left,leftmargin=*,nosep]
    \item if $(a,b) \in S$, $a',b' \in \bbn$ and $\vct{a}{b} - \vct{a'}{b'} = (g,h) \in \subgp{d}{n}$,
    then $g<0$ or $h<0$ or $(g,h)=(0,0)$, and \label{item:not0}
    \item if $(a,b),(a',b') \in S$ and $(a,b) \neq (a',b')$, then $\vct{a}{b} \neq \vct{a'}{b'}$.
  \end{enumerate}
\end{lem}

\begin{proof}
  Every monomial in $(x^d,y^n)$ can be written as a scalar multiple of $x^g\compvec{x}{a}{b}y^h$
  for some $a,b \in \bbn$ and $(g,h) \in \subgp{d}{n}$ with $g,h \geq 0$ and $(g,h) \neq (0,0)$.
\end{proof}

\begin{lem} \label{lem:basis0}
  The set $\basis_0$ is linearly independent in $R/(x^d,y^n)$ over $k$.
\end{lem}

\begin{proof}
  By Lemma~\ref{lem:sizecand}, we only need to verify \ref{item:not0} in Lemma~\ref{lem:indep}
  for $(a,b) \in B_0$ and $(a',b') \notin B_0$. By Remark~\ref{rmk:quadrant}, we may assume that
  $a'<a$ or $b'<b$. If $a'<a$, then by Remark~\ref{rmk:quadrant} we have $b'>b$, so
  $\vct{a-a'}{0} = \vct{0}{b'-b} + (g,h)$. By the minimality of $a_3$ we have $g<0$, $h<0$ or
  $(g,h)=(0,0)$. Similarly, if $b'<b$, then $a'>a$ and $\vct{0}{b-b'} = \vct{a'-a}{0} + (g,h)$
  and the result follows from the minimality of $b_2$.
\end{proof}

\begin{thm} \label{thm:fourgen}
  For the ring $R = k[x^d,x^ey^{\ell},x^fy^m,y^n]$, we have:
  \begin{enumerate}[label=(\roman*),align=left,leftmargin=*,nosep]
    \item \label{item:cramer} $\ds |\basis_0| = |H| =
    \left| \begin{matrix}
      a_3 & -b_3\\
      -a_2 & b_2
    \end{matrix} \right| =
    \left| \begin{matrix}
      a_3 & -b_3\\
      a_1 & b_1
    \end{matrix} \right| =
    \left| \begin{matrix}
      a_1 & b_1\\
      -a_2 & b_2
    \end{matrix} \right|$
    \item \label{item:ind} $\dim_k R/(x^d,y^n) \geq |H|=|\basis_0|$
    \item \label{item:cm} The ring $R$ is Cohen-Macaulay if and only if $\basis_0$ is a basis of $R/(x^d,y^n)$ over
    $k$ if and only if $g_2,h_2 \geq 0$.
  \end{enumerate}
\end{thm}

\begin{proof}
  \ref{item:cramer}: We have $|\basis_0| = |B_0| = |H|$ by Lemma~\ref{lem:sizecand}.
  The definition of $B_0$ gives
  \[
    |B_0| = a_3b_2 - (a_3-a_1)(b_2-b_1) = a_3b_2-a_2b_3
  \]
  The rest again follows from $\eqref{eqn:(d,n)} = \eqref{eqn:(f,m)} + \eqref{eqn:(e,l)}$.
  \separate

  \ref{item:ind}: By Lemma~\ref{lem:basis0}, the set $\basis_0$ is linearly independent in $R/(x^d,y^n)$
  over $k$.
  \separate

  \ref{item:cm}: If $g_2<0$ or $h_2<0$, then $(a,b)=(0,b_2)$ and $(a',b')=(a_2,0)$ satisfy
  Lemma~\ref{lem:indep} by \eqref{eqn:(f,m)}. Let us verify Lemma~\ref{lem:indep}\ref{item:not0} for
  $(0,b_2)$ and $(a',b') \notin B_0$. By Remark~\ref{rmk:quadrant} we may assume that $b' < b_2$.
  If $b'>0$, then $\vct{0}{b_2-b'} = \vct{a'}{0} + (g,h)$ and \ref{item:not0} is satisfied by the linear
  independence of $\basis_0$ in $R/(x^d,y^n)$ over $k$. If $b'=0$, then $a' \geq a_3 > a_2$. If
  $\vct{0}{b_2} = \vct{a'}{0} + (g,h)$ with $g,h \geq 0$, then $g_2,h_2 \geq 0$ in \eqref{eqn:(f,m)},
  contradicting our assumption. Therefore Lemma~\ref{lem:indep}\ref{item:not0} holds for $(a,b)=(0,b_2)$
  and hence $\basis_0 \cup \{\vec{x}^{\vct{0}{b_2}}\}$ is linearly independent in $R/(x^d,y^n)$ over $k$.
  \separate

  If $g_2,h_2 \geq 0$, then $\compvec{x}{0}{b_2} = \compvec{x}{a_2}{0}$ or $\compvec{x}{0}{b_2} \in
  \compvec{x}{a_2}{0}(x^d,y^n)$. In the first case, for $a,b \in \bbn$ we have $\compvec{x}{a}{b}
  = \compvec{x}{a+qa_2}{b-qb_2}$ for any $q \in \bbz$. So by the definition of $B_0$ and
  Remark~\ref{rmk:quadrant} we see that for all $(a',b') \notin B_0$ either $\compvec{x}{a'}{b'} \in
  (x^d,y^n)$ or $\compvec{x}{a'}{b'} =\compvec{x}{a}{b}$ for some $(a,b) \in B_0$.
  Therefore $\basis_0$ is a basis of $R/(x^d,y^n)$ over $k$.
  \separate

  Finally, Theorems~\ref{thm:matsu} and \ref{thm:subgp} show that
  $R$ is Cohen-Macaulay if and only if $\dim_k R/(x^d,y^n) = |H|$. By \ref{item:ind}, $\dim_k R/(x^d,y^n)=|H|$
  if and only if $\basis_0$ is a basis of $R/(x^d,y^n)$ over $k$ if and only if $g_2,h_2 \geq 0$.
\end{proof}

\begin{rmk}
  In part \ref{item:cm} of Theorem~\ref{thm:fourgen}, instead of using Theorems~\ref{thm:matsu}
  and \ref{thm:subgp}, one can also prove the result using the fact that $R$ is Cohen-Macaulay if and only if
  $x^d,y^n$ is a regular sequence.
\end{rmk}

\begin{cor}
  The ring $k[x^d,x^ey^{\ell},y^n]$ is Cohen-Macaulay, where $d,n>0$ and $(e,\ell) \neq (0,0)$.
\end{cor}

\begin{proof}
  Take $(f,m) = u_1(d,0) + u_2(e,\ell) + u_3(0,n)$ for any $u_1,u_2,u_3 \in \bbn$.
\end{proof}

\begin{thm} \label{thm:basis}
  We can use the following algorithm to obtain a basis of $R/(x^d,y^n)$ over $k$.
  \begin{enumerate}[label=\arabic*,align=left,leftmargin=*,nosep]
    \item \label{item:initial} Let $B = B_0$.
    \item Let $base = a_1$, $a^*=a_2$, $b^*=b_2$, $g^*=g_2$ and $h^*=h_2$.
    \item While $g^*<0$ or $h^*<0$, do the following steps.
    \item \label{item:geqa1}
    If $a^* \geq a_1$, then:\\
    Replace $B$ by $B \cup \{ (0,b^*) + (u,v) \mid u < base \text{ and } v < b_1 \}$.\\
    Replace $a^*$ by $a^*-a_1$, $b^*$ by $b^*+b_1$, $g^*$ by $g^*+g_1$ and
    $h^*$ by $h^*+h_1$.
    \item \label{item:lla1}
    If $a^* \leq a_1 - base$, then:\\
    Replace $B$ by $B \cup \{ (0,b^*) + (u,v) \mid u < base \text{ and } v < b_2 \}$.\\
    Replace $a^*$ by $a^*+a_2$, $b^*$ by $b^*+b_2$, $g^*$ by $g^*+g_2$
    and $h^*$ by $h^*+h_2$.
    \item \label{item:lta1}
    If $a_1 - base < a^* < a_1$, then:\\
    Replace $B$ by $B \cup \{ (0,b^*) + (u,v) \mid (u < base \text{ and } v < b_1) \text{ or }
    (u < a_1-a^* \text{ and } v < b_2)\}$.\\
    Replace $a^*$ by $a^*+a_2$, $b$ by $b^*+b_2$, $g^*$ by $g^*+g_2$,
    $h^*$ by $h^*+h_2$ and $base$ by $a_1-a^*$.
  \end{enumerate}
  After the algorithm stops, the set of monomials $\basis = \{ \vec{x}^{\vct{a}{b}} \mid
  (a,b) \in B\}$ forms a basis of $R/(x^d,y^n)$ over $k$.
\end{thm}

\begin{rmk} \label{rmk:algor}
  Theorem~\ref{thm:basis} only needs to use information from \eqref{eqn:(d,n)} and
  \eqref{eqn:(f,m)}, or equivalently, from \eqref{eqn:(f,m)} and \eqref{eqn:(e,l)}. Given the equation
  \begin{equation} \label{eqn:algor}
    -a^*(e,\ell) + b^*(f,m) = (g^*,h^*),
  \end{equation}
  Step~\ref{item:geqa1}\ corresponds to $\eqref{eqn:algor} + \eqref{eqn:(d,n)}$ and
  Steps~\ref{item:lla1}\ and \ref{item:lta1}\ correspond to $\eqref{eqn:algor} + \eqref{eqn:(f,m)}$.
  Furthermore, in each iteration of the algorithm, the new elements added to the set $B$
  are in one-to-one correspondence with those in $\{ (a,b) \in B_0 \mid a^* \leq a < a^* + base \}$.
\end{rmk}

\begin{proof}[Proof of Theorem~\ref{thm:basis}]
  First, we note by induction that throughout the algorithm,
  \begin{enumerate}[label=(\alph*),align=left,leftmargin=*,nosep]
    \item $a^* + base \leq a_3$,
    \item the value of $base$ is always positive and weakly decreasing, and
    \item if $a,b,a',b' \in \bbn$, $a' \leq a$, $b' \leq b$ and $(a,b) \in B$, then $(a',b') \in B$.
  \end{enumerate}
  Let $^u$ denote the updated value of a variable after an iteration of the algorithm. We note also
  that in Steps~\ref{item:geqa1}, \ref{item:lla1} and \ref{item:lta1}:
  \begin{enumerate}[resume,label=(\alph*),align=left,leftmargin=*,nosep]
    \item \label{item:notequiv}
    Let $C = B^u \setminus B$ and $(a,b) \in C$. Then $\vct{a}{b} \equiv \vct{a+a^*}{b-b^*}$
    and $(a+a^*,b-b^*) \in B_0$. Hence if $(a',b') \in C$ such that $(a,b) \neq (a',b')$, then $\vct{a}{b}
    \not\equiv \vct{a'}{b'}$.
  \end{enumerate}
  \separate

  We will now prove by induction on the number of iterations that after
  each iteration of the algorithm,
  \begin{enumerate}[resume,label=(\alph*),align=left,leftmargin=*,nosep]
    \item \label{item:inductind} the set $\basis^u = \{\compvec{x}{a}{b} \mid (a,b) \in B^u\}$ is linearly
    independent in $R/(x^d,y^n)$ over $k$, and
    \item \label{item:is0}
    $\vec{x}^{\vct{a'}{b'}} \in (x^d,y^n)$ for all $(a',b') \notin B^u$ such that $b'<b^{*u}$.
  \end{enumerate}
  \separate

  The base case of $B=\emptyset$, i.e.\ $B^u=B_0$, is given by Theorem~\ref{thm:fourgen}\ref{item:ind}
  and Remark~\ref{rmk:quadrant}. In the induction step, we will first show \ref{item:inductind}
  by using Lemma~\ref{lem:indep}.
  \separate

  Let $(a,b) \in C = B^u \setminus B$ and $(a',b') \in B$ such that $\vct{a}{b} \equiv \vct{a'}{b'}$.
  If $(a',b') \in B_0$, then $(a',b') = (a+a^*,b-b^*)$ and $\vct{a}{b} - \vct{a+a^*}{b-b^*} =
  (g^*,h^*)$. By assumption, $g^*<0$ or $h^*<0$, so $(a,b)$ and $(a',b')$ satisfy
  Lemma~\ref{lem:indep}. If $(a',b') \notin B_0$, then we have $\vct{a}{b-b_2} \equiv
  \vct{a'}{b'-b_2}$ and by the linear independence of $\basis$, $(a,b)$ and $(a',b')$ again satisfy
  Lemma~\ref{lem:indep}.
  \separate

  So suppose that $(a',b') \notin B$ and $\vct{a}{b} - \vct{a'}{b'} = (g,h) \in \subgp{d}{n}$. Let us
  verify Lemma~\ref{lem:indep}\ref{item:not0}. By the proof of Lemma~\ref{lem:basis0}, we may
  assume that $b'<b$.
  \separate

  \textit{Case 1:} $b' \geq b_2$. Lemma~\ref{lem:indep}\ref{item:not0} holds from
  $\vct{a}{b-b'}-\vct{a'}{0} = (g,h)$ and the linear independence of $\basis$.
  \separate

  \textit{Case 2:} $b_2>b'\geq b_1$. By the definition of $B_0$ we have $a' \geq a_1$. Let $q \in \bbn$
  be such that $a'-qa_1, b'-qb_1\geq 0$ but one of $a'-(q+1)a_1$ or $b'-(q+1)b_1$ is negative,
  so that $\vct{a}{b} - \vct{a'-qa_1}{b'-qb_1} = \vct{a}{b} - \vct{a'}{b'} + q\vct{a_1}{b_1} =
  (g+qg_1,h+qh_1)$. If $(a'-qa_1,b'-qb_1) \in B_0$, then $g+qg_1=g^*<0$ or $h+qh_1=h^*<0$,
  so Lemma~\ref{lem:indep}\ref{item:not0} holds for $(a,b)$ and $(a',b')$.
  Otherwise, replacing $(a'-qa_1,b'-qb_1)$ by $(a',b')$, we are reduced to the case where $b'<b_1$.
  \separate

  \textit{Case 3:} $b_1 > b'$. By the definition of $B_0$ we have $a' \geq a_3$. Then $\vct{a'}{b'}
  - \vct{a+a^*}{b-b^*} = \vct{a'}{b'} - \vct{a}{b} + \vct{a}{b} - \vct{a+a^*}{b-b^*}
  = (g^*-g,h^*-h)$. If $b' \geq b-b^*$, then $0>g^* \geq g$ or $0>h^* \geq h$ and
  Lemma~\ref{lem:indep}\ref{item:not0} is satisfied. If $b'<b-b^*$, then $\vct{0}{b-b^*-b'}
  = \vct{a'-(a+a^*)}{0} + (g-g^*,h-h^*)$. By the minimality of $b_2$ we have $g-g^*,h-h^*
  \leq 0$ and again we are done.
  \separate

  Now we verify \ref{item:is0}. Let $(a',b') \notin B^u$ with $b'<b^{*u}$. By induction, we may
  assume that $b' \geq b^*$ and by Remark~\ref{rmk:quadrant} we may assume that $a' < base$,
  so we only need to consider Step~\ref{item:lta1} with $a'\geq a_1-a^*$ and $b'\geq b^*+b_1$.
  We have $\vct{a_1-a^*}{b^*+b_1} = \vct{a_1}{b_1} + (g^*,h^*) \in \subgp{d}{n}$, so
  $\compvec{x}{a_1-a^*}{b^*+b_1} \in (x^d,y^n)$ and the result follows from
  Remark~\ref{rmk:quadrant}.
  \separate

  Finally, the algorithm must stop at or before $b^* = \ord((f,m),\qtgp{d}{n})$. After the algorithm
  stops, we already know that $\basis$ is linearly independent by \ref{item:inductind}. By
  \eqref{eqn:algor} we have $\compvec{x}{0}{b^*} = \compvec{x}{a^*}{0}$ or
  $\compvec{x}{0}{b^*} \in \compvec{x}{a^*}{0}(x^d,y^n)$. By \ref{item:is0} and
  Remark~\ref{rmk:quadrant} we see that for all $(a',b') \notin B$ either $\compvec{x}{a'}{b'} \in
  (x^d,y^n)$ or $\compvec{x}{a'}{b'}=\compvec{x}{a}{b}$ for some $(a,b) \in B$.
  Therefore $\basis$ is a basis of $R/(x^d,y^n)$ over $k$.
\end{proof}

\begin{rmk}
  Each iteration of the algorithm gives $\vct{0}{b^*} \equiv \vct{a^*}{0}$. Since the algorithm
  must stop at or before $b^* = \ord((f,m),\qtgp{d}{n})$, we cannot have $\vct{0}{b^{1*}}
  \equiv \vct{0}{b^{2*}}$ for different values $b^{1*},b^{2*}$ of $b^*$. So the number of
  iterations of the algorithm is at most $a_3 \leq \ord((e,\ell),\qtgp{d}{n}) \leq |H| \leq dn$.
\end{rmk}

\begin{cor} \label{cor:length}
  We have $\dim_k R/(x^d,y^n) \leq |H|(|H|+1)/2 \leq dn(dn+1)/2$.
\end{cor}

\begin{proof}
  Let us set $a^* = 0$ in Step~\ref{item:initial}. In each iteration of the algorithm, at most
  $i_{a^*} = | \{ (a,b) \in B_0 \mid a \geq a^* \} |$ elements are added to the set $B$ by
  Remark~\ref{rmk:algor}. We have $1 \leq i_{a^*} \leq |B_0| = |H| \leq dn$ and that the map
  $a^* \mapsto i_{a^*}$ is injective. Since there exist at most $|H|$ possible values of $a^*$
  before the stopping criterion is reached,
  we have $\dim_k R/(x^d,y^n) = |B| \leq \sum_{i=1}^{|H|} i = |H|(|H|+1)/2 \leq dn(dn+1)/2$.
\end{proof}

\begin{eg}
  Here we give examples showing that the algorithm is ``best possible'', in the sense that
  the maximum number of iterations can be attained. In the second example, we will show that
  the upper bound in Corollary~\ref{cor:length} is also attained.
  Let $p,q$ be distinct prime numbers.
  \begin{enumerate}[label=(\alph*),align=left,leftmargin=*,nosep]
    \item Let $R=k[x^p,x^{jpq-1}y,xy^{jpq-1},y^q]$ with $j \in \bbn$, $j>1$.
    The successive values of $\vct{0}{b^*}$ are
    \[
      (1,jpq-1),\,(2,2(jpq-1)),\,\dots,\,(pq-1,(pq-1)(jpq-1)),\,(pq,pq(jpq-1))
    \]
    and those of $\vct{a^*}{0}$ are
    \[
      ((pq-1)(jpq-1),pq-1),\,((pq-2)(jpq-1),pq-2),\,\dots,\,(jpq-1,1),\,(0,0)
    \]
    When $j=2$, $p=2$ and $q=3$ (or $p=3$ and $q=2$), we display the elements of
    $\langle B \rangle = \{ \vct{a}{b} \mid (a,b) \in B \} = \log(\basis)$ as follows.
    \begin{center}
      \begin{tabular}{c c c c c c}
        (0,0) & (11,1) & (22,2) & (33,3) & (44,4) & (55,5)\\
        (1,11)\\
        (2,22)\\
        (3,33)\\
        (4,44)\\
        (5,55)\\
        \multicolumn{6}{c}{$R=k[x^2,x^{11}y,xy^{11},y^3]$}
      \end{tabular}
    \end{center}\vspace{3pt}
    \item Let $R=k[x^p,x^{jpq+1}y,xy^{jpq+1},y^q]$ with $j \in \bbn$, $j>0$.
    The successive values of $\vct{0}{b^*}$ are
    \[
      (1,jpq+1),\,(2,2(jpq+1)),\,\dots,\,(pq-1,(pq-1)(jpq+1)),\,(pq,pq(jpq+1))
    \]
    and those of $\vct{a^*}{0}$ are
    \[
      (jpq+1,1),\,(2(jpq+1),2),\,\dots,\,((pq-1)(jpq+1),pq-1),\,(0,0)
    \]
    When $j=1$, $p=2$ and $q=3$, we display the elements of $\langle B \rangle$ as follows.
    \begin{center}
      \begin{tabular}{c c c c c c}
        (0,0) & (7,1) & (14,2) & (21,3) & (28,4) & (35,5)\\
        (1,7) & (8,8) & (15,9) & (22,10) & (29,11)\\
        (2,14) & (9,15) & (16,16) & (23,17)\\
        (3,21) & (10,22) & (17,23)\\
        (4,28) & (11,29)\\
        (5,35)\\
        \multicolumn{6}{c}{$R=k[x^2,x^7y,xy^7,y^3]$}
      \end{tabular}
    \end{center}
  \end{enumerate}
\end{eg}

\begin{rmk}
  Having found the basis $\basis$ of $R/(x^a,y^b)$ as in Theorem~\ref{thm:basis}, one may
  sort the monomials in $\basis$ and find the Hilbert polynomial $P(n)$ for $(x^a,y^b) \subseteq R$
  and the least integer $m$ such that the Hilbert polynomial equals the
  Hilbert function for all $n \geq m$ by Theorem~\ref{thm:subgp}.
\end{rmk}

\section{Projective monomial curves in $\bbp^3$}
\label{sec:p3}

In this section, we will consider rings of the form $R=k[x^n, x^{n-\ell}y^{\ell}, x^{n-m}y^m, y^n]$
with $0<\ell<m<n$. We will apply the results from Section~\ref{sec:fourgen} to obtain stronger results
for such rings $R$. In particular, Theorem~\ref{thm:p3} gives a simple criterion to determine whether
$R$ is Cohen-Macaulay and Theorem~\ref{thm:p3basis} gives a simple algorithm to generate a $k$-basis
of $R/(x^n,y^n)$.

\begin{notn}\label{notn:abchagain}
  In this section, we fix $a_i,b_i,c_i,h_i \in \bbn$, $i=1,2,3$ as follows.
  Let $c_1$ be the smallest integer such that there are $m/\!\gcd(\ell,m) \geq a_1 > 0$ and
  $b_1>0$ with
  \begin{equation} \label{eqn:c}
    a_1\ell + b_1 m = c_1 n = h_1
  \end{equation}
  
  Let $b_2$ be the smallest integer such that there are $ n/\!\gcd(\ell,n) > a_2 \geq 0$ and
  $c_2 > 0$ with
  \begin{equation} \label{eqn:b}
    -a_2\ell + b_2 m = c_2 n = h_2
  \end{equation}
  
  Let $a_3$ be the smallest positive integer such that there are $ n/\!\gcd(m,n) > b_3 \geq 0$
  and $c_3 \geq 0$ with
  \begin{equation} \label{eqn:a}
    a_3\ell - b_3 m = c_3 n = h_3
  \end{equation}
\end{notn}

\begin{rmk}
  We recall from Section~\ref{sec:intro} that for any $d \in \bbz$ we have $\ord((n-d,d),\qtgp{n}{n})
  = \ord((-d,d),\qtgp{n}{n}) = n/\!\gcd(d,n)$.
\end{rmk}

\begin{lem} \label{lem:geq}
  Let $a,b,c,d \in \bbn$.
  \begin{enumerate}[label=(\roman*),align=left,leftmargin=*,nosep]
    \item If $-a(n-\ell,\ell)+b(n-m,m)=(cn,dn)$, $b>0$ and $c \geq 0$, then $d > 0$. \label{item:cgeq0}
    \item If $a(n-\ell,\ell)-b(n-m,m)=(cn,dn)$, $a>0$ and $d \geq 0$, then $c > 0$. \label{item:dgeq0}
  \end{enumerate}
\end{lem}

\begin{proof}
  \ref{item:cgeq0}: Since $b>0$ we have $(a,c) \neq (0,0)$. If $c \geq 0$, then $b(n-m) = cn+a(n-\ell) >
  (c+a)(n-m)$, so $b>c+a$. Hence $dn=-a\ell + bm = (b-a-c)n>0$.\\
  \ref{item:dgeq0}: Replace $a$ by $b$, $b$ by $a$, $\ell$ by $n-m$, $m$ by $n-\ell$, $c$ by $d$
  and $d$ by $c$ in \ref{item:cgeq0}.
\end{proof}

\begin{lem}
  The definitions of $a_i,b_i,h_i$, $i=1,2,3$ in Notation~\ref{notn:abch} and \ref{notn:abchagain} coincide.
  In particular, we have $\eqref{eqn:c} = \eqref{eqn:b} + \eqref{eqn:a}$.
\end{lem}

\begin{proof}
  Let us temporarily write $a'_i,b'_i,h'_i$, $i=1,2,3$ for the definitions of $a_i,b_i,h_i$ in this section.
  We have $e=n-\ell$ and $f=n-m$. Let us first consider \eqref{eqn:(f,m)}. Suppose that $g_2 \geq 0$.
  By Lemma~\ref{lem:geq}\ref{item:cgeq0} we have $h_2>0$, so the conditions that $g_2 >0$ or
  $(g_2,h_2)=(0,0)$ become redundant. Hence $b_2=b'_2$, $a_2=a'_2$ and $h_2=h'_2$.
  \separate

  Similarly, in \eqref{eqn:(e,l)} suppose that $h_3 \geq 0$. By Lemma~\ref{lem:geq}\ref{item:dgeq0}
  the conditions $g_3 \geq 0$ and $(g_3,h_3) \neq (0,0)$ are redundant.
  Hence $a_3=a'_3$, $b_3=b'_3$ and $h_3=h'_3$.
  \separate

  Now $\eqref{eqn:b}+\eqref{eqn:a}$ gives
  \[
    (a_3-a_2)\ell + (b_2-b_3)m = (c_2+c_3)n = h_2+h_3
  \]
  Let us show that $c_1=c_2+c_3$. First we have $a_3-a_2 \leq a_3  \leq m/\!\gcd(\ell,m)$.
  Now suppose that $a,b,c\in \bbn$ are such that $a,b>0$, $a\ell + bm = cn$
  and $(a,b) \neq (a_3-a_2,b_2-b_3)$. By Remark~\ref{rmk:quadrant} we may assume that
  $a<a_3-a_2$ or $b<b_2-b_3$. If $b<b_2-b_3$, then $c\geq c_2+c_3$ by Lemma~\ref{lem:123}.
  If $c=c_2+c_3$, then $(a-(a_3-a_2))\ell=(b_2-b_3-b)m$, so $(m/\!\gcd(\ell,m)) \mid a-(a_3-a_2)$
  and $a>m/\!\gcd(\ell,m)$. If $a<a_3-a_2$, then $b>b_2-b_3$ by Lemma~\ref{lem:123} and
  hence $c>c_2+c_3$ by the minimality of $a_3$. Therefore $h_1=h'_1$, $a_1=a'_1$ and $b_1=b'_1$.
\end{proof}

\begin{notn}
  We let $B_0,H,\compvec{x}{a}{b},\basis_0$ be as in Notation~\ref{notn:b0} and \ref{notn:vectorx}.
\end{notn}

\begin{thm} \label{thm:p3}
  Let $R = k[x^n,x^{n-\ell}y^{\ell},x^{n-m}y^m,y^n]$ and $d=\gcd(\ell,m,n)$. Then:
  \vspace{3pt}
  
  \begin{enumerate}[label=(\roman*),align=left,leftmargin=*,nosep]
    \item \label{item:cramern} $\ds n =
    d\left| \begin{matrix}
      a_3 & -b_3\\
      -a_2 & b_2
    \end{matrix} \right| =
    d\left| \begin{matrix}
      a_3 & -b_3\\
      a_1 & b_1
    \end{matrix} \right| =
    d\left| \begin{matrix}
      a_1 & b_1\\
      -a_2 & b_2
    \end{matrix} \right|$ \vspace{3pt}
    \item \label{item:cramerm} $\ds m =
    d\left| \begin{matrix}
      a_3 & c_3\\
      -a_2 & c_2
    \end{matrix} \right| =
    d\left| \begin{matrix}
      a_3 & c_3\\
      a_1 & c_1
    \end{matrix} \right| =
    d\left| \begin{matrix}
      a_1 & c_1\\
      -a_2 & c_2
    \end{matrix} \right|$ \vspace{3pt}
    \item \label{item:cramerl} $\ds \ell =
    d\left| \begin{matrix}
      c_3 & -b_3\\
      c_2 & b_2
    \end{matrix} \right| =
    d\left| \begin{matrix}
      c_3 & -b_3\\
      c_1 & b_1
    \end{matrix} \right| =
    d\left| \begin{matrix}
      c_1 & b_1\\
      c_2 & b_2
    \end{matrix} \right|$ \vspace{3pt}
    \item \label{item:p3ind} $\dim_k R/(x^d,y^n) \geq n/d= |\basis_0|$
    \item \label{item:p3cm} The ring $R$ is Cohen-Macaulay iff $\basis_0$ is a basis of $R/(x^d,y^n)$ over
    $k$ iff $b_2 \geq a_2 + c_2$.
  \end{enumerate}
\end{thm}

\begin{proof}
  We may identify $H$ with the subgroup of $\bbz/n\bbz$ generated by $\ell$ and $m$, so $|H|=
  n/d$. Therefore \ref{item:p3ind} and \ref{item:cramern} follow from
  \ref{item:cramer} of Theorem~\ref{thm:fourgen}, and \ref{item:cramerm} and \ref{item:cramerl}
  follow from Cramer's rule. Using \ref{item:cm} of Theorem~\ref{thm:fourgen}, we have
  $g_2 \geq 0$ iff $-a_2(n-\ell)+b_2(n-m) \geq 0$ iff $(b_2-a_2-c_2)n \geq 0$ iff $b_2 \geq a_2+c_2$,
  so \ref{item:p3cm} follows from Lemma~\ref{lem:geq}\ref{item:cgeq0}.
\end{proof}

\begin{cor}
  Let $\ell=1$ and $n=qm+r$ as in the Euclidean algorithm. If $r=0$, then $R$ is Cohen-Macaulay.
  If $r \neq 0$, then $R$ is Cohen-Macaulay if and only if $q+r \geq m$.
\end{cor}

\begin{proof}
  If $r = 0$, then $a_2=0$, $b_2=q$ and $c_2=1$, so $b_2 = q \geq 1 = a_2+c_2$. If $r \neq 0$,
  then $a_2=m-r$, $b_2=q+1$ and $c_2=1$, so $R$ is Cohen-Macaulay iff $q+1 \geq m-r+1$ iff
  $q+r \geq m$.
\end{proof}

\begin{cor}
  If $\gcd(\ell,m)=1$ and $\ell + m = n$, then $R$ is Cohen-Macaulay if and only if $m = \ell + 1$.
\end{cor}

\begin{proof}
  Since $\gcd(\ell,m)=1$ and $\ell + m = n$ we have $\gcd(m,n)=1$. From $b_2m - a_2(n-m)=c_2n$
  we get $(b_2+a_2)m=(c_2+a_2)n$. By the minimality of $b_2$ we get $b_2+a_2=n$ and
  $c_2+a_2=m$, so $c_2=1$, $a_2=m-1$ and $b_2=n-(m-1)$. Then $b_2 \geq a_2+c_2$ iff
  $n-(m-1) \geq m-1+1$ iff $n \geq 2m-1$. But $n=m+\ell \leq m+m-1 = 2m-1$, so $R$ is
  Cohen-Macaulay iff $n=2m-1$ iff $m=\ell+1$.
\end{proof}

\begin{rmk}
  We therefore recover Macaulay's result that $k[x^4,x^3y,xy^3,x^4]$ is not Cohen-Macaulay.
\end{rmk}

\begin{thm} \label{thm:p3basis}
  We can use the following algorithm to obtain a basis of $R/(x^n,y^n)$ over $k$.
  \begin{enumerate}[label=\arabic*,align=left,leftmargin=*,nosep]
    \item \label{item:p3initial} Let $B = B_0$.
    \item Let $base = a_1$, $a^*=a_2$, $b^*=b_2$ and $c^*=c_2$.
    \item While $b^*<a^*+c^*$, do the following steps.
    \item \label{item:p3geqa1}
    If $a^* \geq a_1$, then:\\
    Replace $B$ by $B \cup \{ (0,b^*) + (u,v) \mid u < base \text{ and } v < b_1 \}$.\\
    Replace $a^*$ by $a^*-a_1$, $b^*$ by $b^*+b_1$ and $c^*$ by $c^*+c_1$.
    \item \label{item:p3lla1}
    If $a^* \leq a_1 - base$, then:\\
    Replace $B$ by $B \cup \{ (0,b^*) + (u,v) \mid u < base \text{ and } v < b_2 \}$.\\
    Replace $a^*$ by $a^*+a_2$, $b^*$ by $b^*+b_2$ and $c^*$ by $c^*+c_2$.
    \item \label{item:p3lta1}
    If $a_1 - base < a^* < a_1$, then:\\
    Replace $B$ by $B \cup \{ (0,b^*) + (u,v) \mid (u < base \text{ and } v < b_1) \text{ or }
    (u < a_1-a^* \text{ and } v < b_2)\}$.\\
    Replace $a^*$ by $a^*+a_2$, $b$ by $b^*+b_2$, $c^*$ by $c^*+c_2$ and $base$ by $a_1-a^*$.
  \end{enumerate}
  After the algorithm stops, the set of monomials $\basis = \{ \vec{x}^{\vct{a}{b}} \mid
  (a,b) \in B\}$ forms a basis of $R/(x^n,y^n)$ over $k$.
  \qed
\end{thm}

\begin{eg}
  Let $R=k[x^{23},x^{21}y^2,x^{5}y^{18},y^{23}]$. We will use Theorem~\ref{thm:p3basis} to
  calculate the size of the monomial $k$-basis $\basis$ of $R/(x^{23},y^{23})$ and
  find the elements of $\basis$.
  \begin{center}
    \begin{tabular}{c | c | c | l c | l}
      Step & $|B|$ & $base$ & \multicolumn{1}{c}{Equation} & \rule{21pt}{0pt}
      & \multicolumn{1}{c}{Remark}\\ \hline
      \rule{0pt}{12pt}    &        &    & \rule{2pt}{0pt} $2 \times 18= -5 \times 2 + 2 \times 23$ & \eqref{eqn:c}\\
      \ref{item:p3initial} &   23 & 5 & \rule{2pt}{0pt} $3 \times 18 = 4 \times 2 + 2 \times 23$ & \eqref{eqn:b}\\
      \ref{item:p3lta1} &     34 & 1 & \rule{2pt}{0pt} $6 \times 18 = 8 \times 2 + 4 \times 23$ & &
      Add equation \eqref{eqn:b} (to equation \eqref{eqn:b}).\\
      \ref{item:p3geqa1} & 36 & 1 & \rule{2pt}{0pt} $8 \times 18 = 3 \times 2 + 6 \times 23$ & &
      Add equation \eqref{eqn:c}.\\
      \ref{item:p3lla1}      & 39 & 1 & $11 \times 18 = 7 \times 2 + 8 \times 23$ & &
      Add equation \eqref{eqn:b}.\\
      \ref{item:p3geqa1} & \fbox{41} & 1 & $13 \times 18 = 2 \times 2 + 10 \times 23$ & &
      Add equation \eqref{eqn:c} and stop.
    \end{tabular}
  \end{center}

  We display the second coordinates of $\langle B \rangle = \log(\basis)$, i.e.\ the $y$-degrees of elements
  of $\basis$, as follows.
  \begin{center}
    \begin{tabular}{r r r r r r r r r}
      0 & 2 & 4 & 6 & 8 & 10 & 12 & 14 & 16\\
      18 & 20 & 22 & 24 & 26 & 28 & 30 & 32 & 34\\
      36 & 38 & 40 & 42 & 44\\
      54 & 56 & 58 & 60 & 62\\
      72 & 74 & 76 & 78 & 80\\
      90\\108\\126\\144\\162\\180\\198\\216
    \end{tabular}
  \end{center}
\end{eg}

\begin{qn}
  Can we find a ``sharp'' bound on $\dim_k R/(x^n,y^n)$ as in Corollary~\ref{cor:length}?
\end{qn}

\end{document}